\documentclass[10pt]{article}
\pdftrailerid{}
\usepackage[utf8]{inputenc}
\usepackage[T1]{fontenc}
\usepackage{amsmath,amssymb,dsfont}
\numberwithin{equation}{section}
\usepackage{microtype}
\usepackage{graphicx,tikz,pgfplots}
\graphicspath{{images/}}
\pgfplotsset{compat=newest}
\usepackage[hyperref,amsmath,thmmarks]{ntheorem}
\usepackage{aliascnt}
\usepackage[a4paper,centering,bindingoffset=0cm,marginpar=2cm,margin=2.5cm]{geometry}
\usepackage[pagestyles]{titlesec}
\usepackage[font=footnotesize,format=plain,labelfont=sc,textfont=sl,width=0.75\textwidth,labelsep=period]{caption}

\usepackage[backend=bibtex,maxnames=10,backref=true,hyperref=true,giveninits=true,safeinputenc]{biblatex}
\DefineBibliographyStrings{english}{%
	backrefpage = {cited on page},
	backrefpages = {cited on pages},
}

\DeclareFieldFormat[report]{title}{``#1''}
\DeclareFieldFormat[book]{title}{``#1''}
\AtEveryBibitem{\clearfield{url}}
\AtEveryBibitem{\clearfield{note}}

\titleformat{\section}[block]{\large\sc\filcenter}{\thesection.}{0.5ex}{}[]
\titleformat{\subsection}[runin]{\bf}{\thesubsection.}{0.5ex}{}[.]

\usepackage[pdftex,colorlinks=true,linkcolor=blue,citecolor=green,urlcolor=blue,bookmarks=true,bookmarksnumbered=true]{hyperref}
\hypersetup
{
    bookmarksnumbered,
    breaklinks,
    pdfborderstyle=,
    colorlinks,
    citecolor=[rgb]{0,.65,0},
    linkcolor=[rgb]{0,0,.5},
    urlcolor=[rgb]{0,0,.5},
    pdfauthor={Authors},
    pdfsubject={Subject},
    pdftitle={Title},
    pdfkeywords={Keywords}
}

\newpagestyle{headers}
{
	\headrule
	\sethead[\footnotesize\thepage][\footnotesize\sc Authors][]{}{\footnotesize\sc Linear Inverse Problems with Coders}{\footnotesize\thepage}
	\setfoot{}{}{}
}
\newtheorem{lemma}{Lemma}[section]

\newaliascnt{proposition}{lemma}

\aliascntresetthe{proposition}

\newaliascnt{corollary}{lemma}

\aliascntresetthe{corollary}

\newaliascnt{theorem}{lemma}
\newtheorem{theorem}[theorem]{Theorem}
\aliascntresetthe{theorem}

\theorembodyfont{\normalfont}
\newaliascnt{definition}{lemma}
\newtheorem{definition}[definition]{Definition}
\aliascntresetthe{definition}

\newaliascnt{assumption}{lemma}

\aliascntresetthe{assumption}

\theorembodyfont{\normalfont}
\newaliascnt{notation}{lemma}
\newtheorem{notation}[notation]{Notation}
\aliascntresetthe{notation}

\theorembodyfont{\normalfont}
\newaliascnt{example}{lemma}
\newtheorem{example}[example]{Example}
\aliascntresetthe{example}

\theorembodyfont{\normalfont}
\newaliascnt{conjecture}{lemma}
\newtheorem{conjecture}[conjecture]{Conjecture}
\aliascntresetthe{conjecture}

\theorembodyfont{\normalfont}
\newaliascnt{remark}{lemma}
\newtheorem{remark}[remark]{Remark}
\aliascntresetthe{remark}

\theoremstyle{nonumberplain}
\theoremseparator{:}
\theoremheaderfont{\normalfont\itshape}

\theoremsymbol{\ensuremath{\square}}
\newtheorem{proof}{Proof}

\newcommand{\N}{\mathds{N}}

\newcommand{\R}{\mathds{R}}

\let\RE\Re
\let\Re=\undefined
\DeclareMathOperator{\Re}{\RE e}
\let\IM\Im
\let\Im=\undefined
\DeclareMathOperator{\Im}{\IM m}

\newcommand{\abs}[1]{\left|#1\right|}
\newcommand{\norm}[1]{\left\|#1\right\|}
\newcommand{\set}[1]{\left\{#1\right\}}
\newcommand{\inner}[2]{\left<#1,#2\right>}
\newcommand{\e}{\mathrm e}
\let\ii\i
\renewcommand{\i}{\mathrm i}

\newcommand{\domain}[1]{\mathcal{D}{(#1)}}
\newcommand{\range}[1]{\mathcal{R}{(#1)}}

\newcommand{\bx}{{\bf{x}}}
\newcommand{\by}{{\bf{y}}}

\newcommand{\bw}{{\bf{w}}}
\newcommand{\bz}{{\bf{z}}}

\newcommand{\bh}{{\bf{h}}}

\newcommand{\vx}{{\vec{x}}}
\newcommand{\vp}{{\vec{p}}\,{}}
\newcommand{\vq}{{\vec{q}}\,{}}
\newcommand{\vh}{{\vec{h}}}
\newcommand{\val}{{\vec{\alpha}}}
\newcommand{\vth}{{\vec{\theta}}}

\newcommand{\be}{{\mathbf{e}}}
\newcommand{\bY}{{\mathbf{Y}}}

\newcommand{\bX}{{\mathbf{X}}}
\newcommand{\bP}{{\vec{P}}}

\newcommand{\bE}{{\mathbf{E}}}

\newcommand{\ve}{\varepsilon}

\title{Newton's methods for solving linear inverse problems with neural network coders%
\footnotetext{\textup{2000} \textit{Mathematics Subject Classification}: 65J22, 47J07}
\footnotetext{\textit{Keywords}: Gauss-Newton's method, inverse problems, neural networks}
}
\author{Otmar Scherzer$^{1,2,3}$\\{\footnotesize\href{mailto:otmar.scherzer@univie.ac.at}{otmar.scherzer@univie.ac.at}}
	\and Bernd Hofmann$^4$\\{\footnotesize\href{mailto:hofmannb@mathematik.tu-chemnitz.de}{hofmannb@mathematik.tu-chemnitz.de}}
	\and Zuhair Nashed$^5$\\{\footnotesize\href{mailto:zuhair.nashed@ucf.edu}{zuhair.nashed@ucf.edu}}
}
\date{\today}

\begin{document}

\maketitle
\thispagestyle{empty}
\begin{center}
\parbox[t]{10em}{\footnotesize
\hspace*{-1ex}$^1$Faculty of Mathematics\\
University of Vienna\\
Oskar-Morgenstern-Platz 1\\
A-1090 Vienna, Austria}
\hfil
\parbox[t]{16em}{\footnotesize
\hspace*{-1ex}$^2$Johann Radon Institute for Computational\\
\hspace*{1em} and Applied Mathematics (RICAM)\\
Altenbergerstraße 69\\
A-4040 Linz, Austria}
\hfil
\parbox[t]{16em}{\footnotesize
	\hspace*{-1ex}$^3$Christian Doppler Laboratory for\\
	\hspace*{1em}Mathematical Modeling and Simulation\\
	\hspace*{1em}of Next Generations of Ultrasound\\
	\hspace*{1em} Devices (MaMSi)\\
	Oskar-Morgenstern-Platz 1\\
	A-1090 Vienna, Austria}
\end{center}
\begin{center}
\parbox[t]{17em}{\footnotesize
\hspace*{-1ex}$^4$Faculty of Mathematics\\
Chemnitz University of Technology\\
Reichenhainer Str. 39/41\\
D-09107 Chemnitz, Germany}
\hfil
	\parbox[t]{15em}{\footnotesize
		\hspace*{-1ex}$^5$College of Sciences\\
		\hspace*{0em}University of Central Florida\\
		\hspace*{0em}4393 Andromeda Loop N\\
		Orlando, FL 32816, USA}
\end{center}

\begin{abstract}
	Neural networks functions are supposed to be able to encode the desired solution of an inverse problem 
	very efficiently.  
	In this paper, we consider the problem of solving linear inverse problems with neural network coders. 
	First we establish some correspondences of this formulation with existing concepts in regularization theory, 
	in particular with state space regularization, operator decomposition and 
	iterative regularization methods. A Gauss-Newton's method is suitable for solving encoded linear inverse problems, 
	which is supported by a local convergence result. 
	The convergence studies, however, are not complete, and are based on a conjecture on linear independence of 
	activation functions and its derivatives.
	\bigskip\par\noindent
	This paper is dedicated to Heinz Engl's 70th birthday.
\end{abstract}

\section{Introduction}
We start the discussion by considering a \emph{general nonlinear operator equation} 
\begin{equation}\label{eq:ip_general}
	N(\vp) = \by,
\end{equation}
where $N:\bP \to \bY$ is a nonlinear operator between Hilbert spaces $\bP$ and $\bY$. 
Particular emphasis is placed on the case when the numerical solution of \autoref{eq:ip_general} is ill-posed or ill-conditioned. 
In this situation regularization methods need to be implemented (see \cite{EngHanNeu96}) 
to guarantee stable solvability. Essentially two classes of methods exist as the basis for solving inverse problems numerically, 
which are \emph{variational methods} 
(see for instance \cite{EngHanNeu96,SchGraGroHalLen09,Hof86,SchuKalHofKaz12,BakGon94,Mor84}) and \emph{iterative methods} 
(see for instance \cite{KalNeuSch08,BakGon89}). 

Modern \emph{coding theory} (see for instance \cite{BorJalPriDim17,DufCamEhr21_report,BraRajRumWir22_report,RifMesVinMulBen11,BorJalPriDim17})
assumes that \emph{natural images} can be represented efficiently by a combination of \emph{neural network} functions. 
Therefore the \emph{set of natural images} is given by the range of a nonlinear mapping $\Psi$,
which maps neural network coefficients to images. In coding theory very often the 
considered operator equation is assumed to be \emph{linear} and \emph{ill-posed}. Let us consider 
therefore the equation
\begin{equation}\label{eq:ip}
	F \bx = \by,
\end{equation}
where $F:\bX \to \bY$ is a bounded linear operator with non-closed range mapping between the Hilbert spaces $\bX$ and $\bY$ such that
$\bx$ and $\by$ are elements of $\bX$ and $\bY$, respectively. 
Combining this with the assumption that the solution of \autoref{eq:ip} is a natural image, or in other words that it can be represented as a 
combination of neural network functions, we get the operator equation 
\begin{equation}\label{eq:ip_2}
	N(\vp)= F \Psi(\vp) = \by,
\end{equation}
where $\Psi:\bP \to \bX$ is the before mentioned nonlinear operator that maps neural network parameters to image functions. We call
\begin{itemize}
	\item $\bX$ the \emph{image space} and 
	\item $\bY$ the \emph{data space}, in accordance with a terminology, which we generated in \cite{AspKorSch20,AspFriKorSch21}.
	\item $\bP$ is called the \emph{parameter space}. We make a different notation for $\bP$, because it represents parametrizations, and is often considered a space of vectors below.
\end{itemize}
The advantage of this ansatz is that the solution of \autoref{eq:ip} is sparsely coded. 
However, the price to pay is that the reformulated \autoref{eq:ip_2} is nonlinear. 
Operator equations of the form of \autoref{eq:ip_2} are not new: 
They have been studied in abstract settings for instance in the context of 
\begin{itemize}
	\item \emph{state space regularization} \cite{ChaKun93} and 
	\item in the context of the \emph{degree of ill--posedness} \cite{GorHof94,Hof99,Hof94,HofSch94,HofTau97} as well as of the \emph{degree of nonlinearity} \cite{HofSch94} of nonlinear ill--posed operator equations.
	\item Another related approach is \emph{finite dimensional approximation} of regularization in Banach spaces (see for instance \cite{PoeResSch10}). 
	Particularly, finite dimensional approximations of regularization methods with neural network functions (in the context of frames and 
	\emph{deep synthesis regularization}) have been studied in \cite{ObmSchwHal21}.
\end{itemize}
In this paper we aim to link the general regularization of the degree of ill--posedness and nonlinearity with coding theory. 
We investigate generalized Gauss-Newton's methods for solving \autoref{eq:ip_2}; Such methods replace the inverse of standard Newton's method by approximations of outer inverses (see \cite{NasChe93}).

The outline of the paper is as follows: In \autoref{sec:dec} we review two decomposition cases as stated in \cite{Hof94} first, 
one of them is \autoref{eq:ip_2}. The study of decomposition cases follows the work on classifying inverse problems and regularization (see \cite{Nas87b}). For operators associated to \autoref{eq:ip_2}, Newton's methods seem to be better suited than gradient descent methods, which we support by a convergence analysis (see \autoref{sec:NN}). Section~\ref{sec:cond} is devoted to solving \autoref{eq:ip_2}, where $\Psi$ is a shallow neural network 
synthesis operator.

\section{Decomposition cases} \label{sec:dec}

We start with a definition for nonlinear operator equations possessing forward operators that are compositions of a linear and a nonlinear operator. Precisely, we distinguish between a first decomposition case (i), where the linear operator is the inner operator and the nonlinear is the outer one,
and a second decomposition case (ii), where the nonlinear operator is the inner operator and the linear is the outer operator.

\begin{definition}[Decomposition cases]
	Let $\bP,\bX, \bY$ be Hilbert-spaces. 
	\begin{enumerate}
		\item 
		An operator $N$ is said to satisfy the \emph{first decomposition case} in an open, non-empty neighborhood 
		$\mathcal{B}(\vp^\dagger;\rho) \subseteq \bP$ of some point $\vp^\dagger$ if there exists a linear operator $F:\bP \to \bX$ and a nonlinear operator $\Psi:\bX \to \bY$ such that 
		\begin{equation*}
			N(\vp) = \Psi(F\vp) \text{ for } \vp \in \mathcal{B}(\vp^\dagger;\rho).
		\end{equation*}
		\item 
		$N$ is said to satisfy the \emph{second decomposition case} in a neighborhood 
		$\mathcal{B}(\vp^\dagger;\rho) \subseteq \bP$ of some point $\vp^\dagger$ if there exists a linear operator $F:\bX \to \bY$
		and a nonlinear operator $\Psi:\bP \to \bX$ such that 
		\begin{equation} \label{eq:decomposition}
			N(\vp) = F \Psi(\vp) \text{ for } \vp \in \mathcal{B}(\vp^\dagger;\rho).
		\end{equation}
		Typically it is assumed that the nonlinear operator $\Psi$ is well--posed.
	\end{enumerate}
\end{definition}
\begin{remark}[First decomposition case] In \cite{Hof94}, this decomposition case has been studied under structural conditions, relating the second derivative of $N$ with the first derivative. Under such assumptions convergence rates conditions (see \cite[Lemma 4.1]{Hof94}) could be proven. The first decomposition case
also arises in inverse option pricing problems in math finance (see \cite{HeiHof03} and \cite[Sect.4]{HofKalPoeSch07}), where the ill-posed compact linear integration operator occurs
as inner operator and a well-posed Nemytskii operator as outer operator.
\end{remark}

\begin{remark}[Second decomposition case]
	Regularization methods for solving operator equations with operators satisfying the second order decomposition case, see \autoref{eq:decomposition}, were probably first analyzed in \cite{ChaKun93} under the name of \emph{state space regularization}. They considered for instance Tikhonov-type regularization methods, consisting in minimization of 
	\begin{equation}\label{eq:CKCS}
		J_\lambda(\vp) = \norm{F\Psi(\vp) -\by}_\bY^2 +\lambda \norm{\Psi(\vp)-\tilde{\bx}}_\bX^2, 
	\end{equation}
	where $\tilde{\bx}$ is a prior and $\lambda > 0$ is a regularization parameter. In \cite{ChaKun93} they derived estimates for the second derivative of $J_\lambda(\vp\!_\lambda)(\bh,\bh)$, where $\bh \in \bP$, that is for the curvature of $J_\lambda$. If the curvature can be bounded from below by some terms $\norm{\bh}_\vp\!\!\!\!{}^2$ then, for instance, a locally unique minimizer of $J_\lambda$ can be guaranteed and also domains can be specified where the functional is convex. Conditions, which guarantee convexity are called 
	\emph{curvature to size conditions}.
	Subsequently, these decomposition cases have been studied exemplarily in \cite{Hof94}. The theory developed there directly applies to \autoref{eq:ip_2}. 
	
	Instead of $J_\lambda$ researchers often study direct regularization with respect to $\vp$. For instance in \cite{DufCamEhr21_report} 
	functionals of the form 
	\begin{equation}\label{eq:CKCSII}
		J_\lambda(\vp) = \norm{F\Psi(\vp) -\by}_\bY^2 +\lambda \mathcal{L}(\vp), 
	\end{equation}
	where $\mathcal{L}$ is some functional directly regularizing the parameter space. Typically $\mathcal{L}$ is chosen to penalize for sparsity of parameters. 
	The main difference between \autoref{eq:CKCS} and \autoref{eq:CKCSII} is that in the prior regularization is performed with respect to the image space $\bX$ and in the later with respect to the parameter space $\bP$. Well-posedness of the functional $J_\lambda$ in \autoref{eq:CKCS} follows if $F \circ \Psi$ is lower-semicontinuous, which in turn follows if $\Psi$ is invertible. 
\end{remark}
In the following we study the solution of decomposable operator equations, such as \autoref{eq:ip_2}, with Gauss-Newton's methods. 
Decomposition cases have been used in the analysis of iterative regularization methods as well (see \cite{KalNeuSch08}):
\begin{definition}[Strong tangential cone condition]
	Let $N: \domain{N} \subset \bP \to \bY$ with $\domain{N}$ its domain be a nonlinear operator. 
	\begin{enumerate}
		\item Then $N$ is said to satisfy the strong tangential cone condition, originally introduced in \cite{HanNeuSch95}, if 
		\begin{equation} \label{eq:stcc}
			N'(\vp\!_2) = R_{\vp\!_2,\vp\!_1} N'(\vp\!_1) \text{ for all } \vp\!_1,\vp\!_2 \in \domain{N}.
		\end{equation}
		where 
		\begin{equation} \label{eq:stcc_norm}
			\norm{R_{\vp\!_2,\vp\!_1}-I} \leq C_T \norm{ \vp\!_2 - \vp\!_1}_\bP.
		\end{equation}
	\item In \cite{Bla96} the \emph{order reversed} tangential cone condition, 
	\begin{equation} \label{eq:stcc_bb}
		N'(\vp\!_2) = N'(\vp\!_1) R_{\vp\!_2,\vp\!_1}  \text{ for all } \vp\!_1,\vp\!_2 \in \domain{N},
	\end{equation}	
	together with \autoref{eq:stcc_norm}, has been introduced.
	\end{enumerate}
\end{definition}

\begin{remark}
	\autoref{eq:stcc} has been used for analyzing \emph{gradient descent methods} (see for instance \cite{HanNeuSch95,KalNeuSch08}). For the analysis of Newton's methods \autoref{eq:stcc_bb} has been used (see \cite{Bla96,KalNeuSch08}).  
\end{remark}
The relation to the decomposition cases is as follows:
\begin{lemma} \label{le:rtcc}
	Let $N: \domain{N} \subseteq \bP \to \bY$ with $\domain{N} = \mathcal{B}(\vp^\dagger;\rho)$
	satisfy the second decomposition case and assume that $\Psi'(\vp)$ is invertible for $\vp \in \domain{N}$. Then $N$ satisfies \autoref{eq:stcc_bb}.
\end{lemma}
\begin{proof}
	If $N$ satisfies the second decomposition case, \autoref{eq:decomposition}, then $N'(\vp) = F \Psi'(\vp)$ for all $\vp \in \domain{N}$. Note, that because $F$ is defined on the whole space $\bX$, $\domain{N}=\domain{\Psi}$.
	By using the invertability assumption on $\Psi$ we get  
	\begin{equation*}
		\begin{aligned}
			N'(\vp\!_2) &= F \Psi'(\vp\!_2) = F \Psi'(\vp\!_1) \underbrace{\Psi'(\vp\!_1)^{-1} \Psi'(\vp\!_2)}_{=:R_{\vp\!_2,\vp\!_1}} = N'(\vp\!_1) R_{\vp\!_2,\vp\!_1},
		\end{aligned}
	\end{equation*}
	which gives the assertion.
\end{proof}
As we have shown, decomposition cases have been extensively studied in the regularization literature. One conclusion out of these studies is that the order reversed 
tangential cone condition \autoref{eq:stcc_bb} is suitable for analyzing Newton's methods \cite{Bla96,KalNeuSch08} and thus in turn for the coded linear operator \autoref{eq:ip_2} because of \autoref{le:rtcc}. The standard tool for analyzing Newton's methods is the Newton-Mysovskii condition as discussed below.

\section{The Newton-Mysovskii Conditions} \label{sec:NN}
In this section we put abstract convergence conditions for Newton type methods in context with decoding. 
We consider first Newton's methods for solving the \emph{general} operator \autoref{eq:ip_general}.
Decomposition cases of the operator $N$ will be considered afterwards.

\subsection{Newton's method with invertible linearizations}
For Newton's methods \emph{local convergence} is guaranteed under \emph{Newton-Mysovskii} conditions. For comparison reasons, we first recall a simple Newton's method analysis in finite dimensional spaces if the nonlinear operator has derivatives which are invertible. The proof of more general results, such as \autoref{th:deupot92dg} below, applies here as well, and thus here the proof is omitted.
Several variants of Newton-Mysovskii conditions have been proposed in the literature (see for instance \cite{DeuHei79,DeuPot92,NasChe93}). 
Analysis of Newton's method was an active research area in the last century, see for instance \cite{Ort68,Schw79}.

\begin{theorem}[Finite dimensional Newton's method] \label{th:deupot92} Let $N: \domain{N} \subseteq \R^n \to \R^n$ be continuously Fr\'echet-differentiable on a non-empty, open and convex set $\domain{N}$. Let $\vp^\dagger \in \domain{N}$ be a solution of \autoref{eq:ip_general}.
	Moreover, we assume that
	\begin{enumerate}
		\item $N'(\vp)$ is invertible for all $\vp \in \domain{N}$ and that 
		\item the \emph{Newton-Mysovskii condition} holds: That is, there exist some $C_N > 0$ such that
		\begin{equation} \label{eq:wrnmi}
			\begin{aligned}
				\norm{N'(\vq)^{-1}(N'(\vp+s(\vq-\vp))-N'(\vp))(\vq-\vp)}_{\bP} \leq s C_N \norm{\vp-\vq}_\bP^2 \\ 
				\text{ for all } \vp, \vq \in \domain{N}, s \in [0,1].
			\end{aligned}
		\end{equation}
	\end{enumerate}	
	Let $\vp^0 \in \domain{N}$ which satisfies
	\begin{equation}\label{eq:h} 
		\overline{\mathcal{B}(\vp^0;\rho)} \subseteq \domain{N} \text{ with }\rho := \norm{\vp^\dagger-\vp^0}_\bP
		\text{ and } h:= \frac{\rho C_I C_L}{2} <1. 
	\end{equation}
	Then the Newton's iteration with starting point $\vp^0$, 
	\begin{equation} \label{eq:newton_invert} \begin{aligned}
			\vp^{k+1} = \vp^k - N'(\vp^k)^{-1}(N(\vp^k)-\by) 
			\quad k \in \N_0,
		\end{aligned}
	\end{equation}
	satisfies that the iterates $\set{\vp^k: k=0,1,2,\ldots}$ belong to $\overline{\mathcal{B}(\vp^0;\rho)}$ and 
	converge quadratically to $\vp^\dagger \in  \overline{\mathcal{B}(\vp^0,\rho)}$.
\end{theorem}
Now, we turn to the case that $N$ is a decomposition operator.

\subsection{Newton-Mysovskii conditions with composed operator}
Now, we study the case of convergence of Gauss-Newton's methods where $N: \bP \to \bY$ with $\bP = \R^{n_*}$ and $\bY$ is an infinite dimensional Hilbert space, where $F:\bX \to \bY$ is linear and bounded and $\Psi:\bP = \R^{n_*} \to \bX$. In this case the Moore-Penrose inverse, or even more general the outer inverse, replaces the inverse in a classical Newton's method (see \autoref{eq:newton_invert}), because linearizations of $N$ will not be invertible anymore as a simple count of dimensions show. We refer now to Gauss-Newton's methods if the linearizations might not be invertible to distinguish between classical Newton's methods also by name. 

Before we phrase a convergence result for Gauss-Newton's methods we recall and introduce some definitions:
\begin{notation}[Inner, outer and Moore-Penrose inverse] \label{not:inverse} (see \cite{Nas76,Nas87}) Let $L: \bP \to \bY$ be a linear and bounded operator mapping between two vector spaces $\bP$ and $\bY$. Then 
	\begin{enumerate}
		\item the operator $B: \bY \to \bP$ is called a \emph{left inverse} to $L$ if 
		$$ B L = I\;. $$
		\item $B:\bY \to \bP$ is called a \emph{right inverse} to $L$ if 
		$$ L B = I\;. $$
		Left and right inverses are used in different context: 
		\begin{itemize}
			\item For a left inverse the nullspace of $L$ has to be trivial, in contrast to $B$.
			\item For a right inverse the nullspace of $B$ has to be trivial.
		\end{itemize}
		\item $B:\bP \to \bP$ is called a \emph{inverse} to $L$ if $B$ is a right and a left inverse.
		\item $B:\bP \to \bY$ is an \emph{outer inverse} to $L$ if 
		\begin{equation} \label{eq:outer}
			BLB = B.
		\end{equation}
		\item Let $\bP$ and $\bY$ be Hilbert-spaces, $L: \bP \to \bY$ be a linear bounded operator.  We denote the orthogonal projections
		$P$ and $Q$ onto $\mathcal{N}(L)$, the nullspace of $L$ (which is closed), and $\overline{\range{L}}$, the closure of the range of $L$:
		That is for all $\vp \in \bP$ and $\by \in \bY$ we have
		\begin{equation} \label{eq:proj}
			P\vp = \text{argmin} \set{\norm{\vp\!_1-\vp}_\bP: \vp\!_1 \in \mathcal{N}(L)} \text{ and }  
			Q\by = \text{argmin} \set{\norm{\by_1-\by}_\bY: \by_1 \in \overline{\mathcal{R}(L)}}. 
		\end{equation}
        We therefore have
		\begin{equation*}
			\begin{aligned}
				P: \bP & \to \mathcal{N}(L) \dot{+} \mathcal{N}(L)^\bot\,\\
				\vp &\mapsto P \vp + 0
			\end{aligned} \quad \text{ and }\quad
			\begin{aligned}
				Q: \bY & \to \overline{\range{L}} \dot{+} \range{L}^\bot. \\
				\by &\to Q\by + 0
			\end{aligned} 
		\end{equation*}
		$B:\domain{B} \subseteq \bY \to \bP$ with $\domain{B}:= \range{L} \dot{+} \range{L}^\bot$ is called the \emph{Moore-Penrose inverse} of $L$ if the following identities hold
		\begin{equation} \label{eq:MP} 
			\begin{aligned}
				LBL &=L,\\
				BLB &=B,\\
				BL &=I-P,\\
				LB &= Q|_{\domain{B}}.
			\end{aligned}
		\end{equation}
	\end{enumerate}	
\end{notation}
In coding theory it is often stated that the range of a neural network operator $\Psi$ forms a manifold in $\bX$, a space, which contains the natural images. 
This is the basis of the following definition making use of the Moore-Penrose inverse.
\begin{definition}[Lipschitz-differentiable immersion] \label{de:immersion} Let $\Psi: \domain{\Psi} \subseteq \bP =\R^{n_*}\to \bX$
	where $\domain{\Psi}$ is open, non-empty, convex and $\bX$ is a seperable (potentially infinite dimensional) Hilbert-space. 
	\begin{enumerate}
		\item \label{it1:immersion}
		We assume that $\mathcal{M}:=\Psi(\domain{\Psi})$ is a $n_*$-dimensional \emph{submanifold} in $\bX$: 
		\begin{itemize}
			\item Let for all $\vp = (p_i)_{i=1}^{n_*} \in \domain{\Psi}$ denote with $\Psi'(\vp)$  the Fr\'echet-derivative of $\Psi$:
			\begin{equation*}
				\begin{aligned}
					\Psi'(\vp): \bP & \to \bX,\\
					\vq = (q_i)_{i=1}^{n_*} &\mapsto \begin{pmatrix}
						\partial_{p_i} \Psi(\vp)
					\end{pmatrix}_{i=1,\ldots,n_*} \vq.
				\end{aligned}
			\end{equation*}
			Here $\begin{pmatrix} \partial_{p_i} \Psi(\vp) \end{pmatrix}_{i=1,\ldots,n_*}$ denotes the vector of functions consisting of all partial derivatives with respect to $\vp$. In differential geometry notation this coincides with the \emph{tangential mapping} $T_{\vp} \Psi$. However, the situation is slightly different here  because $\bX$ can be infinite dimensional.
			
			\item The \emph{representation mapping} of the derivative
			\begin{equation*}
				\begin{aligned}
					\Psi': \domain{\Psi} & \to \bX^{n_*},\\
					\vp &\mapsto \begin{pmatrix}
						\partial_{p_i} \Psi(\vp)
					\end{pmatrix}_{i=1,\ldots,n_*}.
				\end{aligned}
			\end{equation*}
			has always the same rank $n_*$ in $\domain{\Psi}$, meaning that all elements of $\partial_{\vp} \Psi(\vp)$ are linearly independent. 
			This assumption means, in particular, that $\Psi$ is an \emph{immersion} and $\mathcal{M}$ is a submanifold. 
		\end{itemize}
		\item \label{it2:immersion} We define 
		\begin{equation} \label{eq:bP} \begin{aligned}
				P_\vp : \bX &\to \bX_\vp:=\text{span}\set{\partial_{p_i} \Psi(\vp): i=1,\ldots,n_*},\\
				\bx &\mapsto P_\vp \bx := \text{argmin} \set{\norm{\bx_1-\bx}_\bX : \bx_1 \in \bX_\vp}
		\end{aligned} \end{equation}
		as the projection from 
		$$\bX = \bX_\vp \dot{+} 
		\bX_\vp\!\!\!\!{}^\bot$$ 
		onto $\bX_\vp$, which is well-defined by the closedness of the finite dimensional subspace $\bX_\vp$.
		
		Next we define the inverse of $\Psi'(\vp)$ on $\bX_\vp$:
		\begin{equation*}
			\begin{aligned}
				\Psi'(\vp)^{-1} : \text{span} \set{\partial_{p_i} \Psi(\vp):i=1,\ldots,n_*}  \to \bP,\\
				\bx = \sum_{i=1}^{n_*} x_i \partial_{p_i} \Psi(\vp) \mapsto (x_i)_{i=1}^{n_*}
			\end{aligned} 
		\end{equation*} 
		and consequently on $\bX$
		\begin{equation} \label{eq:MP_Penrose}
			\begin{aligned}
				\Psi'(\vp)^\dagger : \bX = \bX_\vp \dot{+} \bX_\vp\!\!\!\!{}^\bot \to \bP,\\
				\bx = (\bx_1,\bx_2) \mapsto \Psi'(\vp)^{-1} \bx_1
			\end{aligned}
		\end{equation}
		which are both well-defined because we assume that $\Psi$ is an immersion.
		Note that $x_i$, $i=1,\ldots,n_*$ are not necessarily the coordinates with respect to an orthonormal system in $\text{span} \set{\partial_{p_i} \Psi(\vp):i=1,\ldots,n_*}$.
		\item \label{it3:immersion} Finally, we assume that the operators $\Psi'(\vp)$ 
		are locally bounded and locally Lipschitz-continuous in $\domain{\Psi}$.
		That is 
		\begin{equation} \label{eq:cl}
			\begin{aligned}
				\norm{\Psi'(\vp)-\Psi'(\vq)}_{\bP \to \bX}
				\leq C_L \norm{\vp-\vq}_{\bP} \quad 
				\norm{\Psi'(\vp)}_{\bP \to \bX}
				\leq C_I \text{ for } \vp, \vq \in \domain{\Psi}.	
			\end{aligned}
		\end{equation}
		If $\Psi$ satisfies these three properties we call it a \emph{Lipschitz-differentiable immersion}.
	\end{enumerate}
\end{definition}
The following lemma is proved by standard means:
\begin{lemma} \label{le:MPN} For a Lipschitz-differentiable immersion 
	\begin{itemize}
		\item the function $\Psi'(\vp)^\dagger: \bX \to \bP$ is in fact the Moore-Penrose inverse of $\Psi'(\vp)$ and 
		\item for every point $\vp \in \domain{\Psi} \subseteq \bP$ there exists a non-empty closed neighborhood where $\Psi'(\vp)^\dagger$ 
		is uniformly bounded and it is Lipschitz-continuous; That is		
		\begin{equation} \label{eq:clb}
			\begin{aligned}
					\norm{\Psi'(\vp)^\dagger-\Psi'(\vq)^\dagger}_{\bX \to \bP} 
				\leq C_L \norm{\vp-\vq}_{\bP}, \quad 
				\norm{\Psi'(\vp)^\dagger}_{\bX \to \bP}
				\leq C_I \text{ for } \vp, \vq \in \domain{\Psi}.	
			\end{aligned}
		\end{equation}
		\item Moreover, the operator $P_\vp$ from \autoref{eq:bP} is bounded. 
	\end{itemize}
\end{lemma}
\begin{proof} 
	\begin{itemize}
		\item 
	We verify the four conditions \autoref{eq:MP} with 
	\begin{itemize}
		\item $L = \Psi'(\vp): \bP = \R^{n_*} \to \bX$, $B = \Psi'(\vp)^\dagger : \bX \to \bP$, with  $\domain{B}=\domain{\Psi'(\vp)^\dagger} = \bX$ and 
		\item $P: \bP \to \bP$ the zero-operator and $Q=P_\vp:\bX \to \bX_\vp$, the projection operator onto $\bX_\vp$	
		(see \autoref{eq:bP}).
	\end{itemize}
	\begin{itemize}
		\item First we prove the third identity with $P=0$ in \autoref{eq:MP}: This follows from the fact that for all $\vq=(q_i)_{i=1}^{n_*} \in \bP$ we have 
	\begin{equation} \label{eq:id1}
		\Psi'(\vp)^\dagger\Psi'(\vp) \vq = \Psi'(\vp)^{-1} \left( \sum_{i=1}^{n_*} q_i \partial_{p_i} \Psi(\vp) \right) = (q_i)_{i=1}^{n_*} = \vq.
	\end{equation}
	\item For the forth identity we see that for all $\bx =(\bx_1,\bx_2) \in \bX$ there exists $x_i$, $i=1,\ldots,n_*$ (because 
	$\partial_{p_i} \Psi(\vp)$, $i=1,\ldots,n_*$ is a basis) such that
	$$ \bx = \sum_{i=1}^{n_*} x_i \partial_{p_i} \Psi(\vp) + \bx_2 \text{ with } \bx_2 \in \bX_\vp\!\!\!\!{}^\bot$$
	and thus
	$$P_\vp \bx =  \sum_{i=1}^{n_*} x_i \partial_{p_i} \Psi(\vp) $$ and therefore 
	\begin{equation} \label{eq:id1a}
	\Psi(\vp)^\dagger \bx = (x_i)_{i=1}^{n_*} = \vx.
	\end{equation}
	Consequently, we have 
	\begin{equation}\label{eq:id2}
		\Psi'(\vp)\Psi'(\vp)^\dagger\bx = \Psi'(\vp) \vx  =P_\vp \bx .
	\end{equation}
	\item 
	For the second identity we use that for all $\bx \in \bX$ 
	\begin{equation} \label{eq:id2a}
		\begin{aligned}
		\Psi'(\vp)^\dagger\Psi'(\vp)\Psi'(\vp)^\dagger\bx &\!\!\!\! \!\underbrace{=}_{\autoref{eq:id1a}} \!\!\!\! \!\Psi'(\vp)^\dagger \Psi'(\vp) \left( \vx \right) \!\!\!\! \! \underbrace{=}_{\autoref{eq:id1}} \!\!\!\! \! \vx
		&\underbrace{=}_{\autoref{eq:id1a}} \Psi(\vp)^\dagger \bx.
		\end{aligned}
	\end{equation}
	\item The first identity is proven analogously. 
\end{itemize}Thus the Moore-Penrose inverse exists.
	\item For the proof of the boundedness we argue as follows:
	First note that
	\begin{equation*}
		\begin{aligned}
			\norm{\Psi'(\vp)^\dagger}_{\bX \to \bP} &= \sup_{\set{\bx \in \bX: \norm{\bx}_\bX=1}} \norm{\Psi'(\vp)^\dagger \bx}_{\bP} 
			=  \sup_{\set{\bx_1 \in \bX_\vp: \norm{\bx_1}=1}} \norm{\Psi'(\vp)^{-1} \bx_1}_{\bP} \\
			&= \sup_{\set{\bx_1 \in \bX_\vp: \norm{\bx_1}=1}} \norm{ \vx_1}_{\bP},
		\end{aligned}
	\end{equation*}
where we used the representation 
\begin{equation} \label{eq:z}
\bx_1 = \sum_{i=1}^{n_*} x_i \be_i \in \bX_\vp \text{ with } \bE := (\be_1=\partial_{p_1} \Psi(\vp),\ldots,\be_{n_*}=\partial_{p_{n_*}} \Psi(\vp)).
\end{equation} 
    It remains to derive a uniform bound and the continuity for $\norm{\vx_1}_{\bP}$ with respect to variations of
    $\vp$. For this purpose we use 
    the Gram-Schmidt inductive procedure to obtain the QR-decomposition of the function valued matrix $\bE$.

    \begin{itemize}
    	\item Because of the assumption \autoref{eq:cl}, each vector $\be_j$, $j=1,\ldots,n_*$, depends continuously of $\vp$.
    	\item The Gram-Schmidt procedure performs only additions, multiplications and division and inner products of vectors of functions, which are all locally Lipschitz-continuous with respect to perturbations as long as the vectors are linearly independent. Thus by chain-rule Gram-Schmidt is Lipschitz-continuous in the matrix entries as long as the vectors $\be_j$, $j=1,\ldots,n_*$ are linearly dependent. 
    	Gram-Schmidt produces an orthogonal family $\hat{\be}_j$, $j=1,\ldots,n_*$, which therefore are also
    	 Lipschitz-continuous with resepct to $\vp$.
    	\item 
    Calculating the QR-decomposition of $\bE$ we get 
    \begin{equation*}
    	\bx_1 = \vx^T \bE = \vx^T Q R \text{ where } R = \begin{pmatrix}
    		\inner{\hat{\be}_1}{\be_1}_\bY & \cdots & \cdots & \cdots & \inner{\hat{\be}_1}{\be_n}_\bY\\
    		0 & \inner{\hat{\be}_2}{\be_2}_\bY & \cdots & \cdots & \inner{\hat{\be}_2}{\be_n}_\bY\\
    		\vdots & \ddots & \ddots & & \\
    		\vdots & & \ddots & \ddots & \\
    		0 & \cdots & \cdots & 0 &  \inner{\hat{\be}_n}{\be_n}_\bY
    	\end{pmatrix},
    \end{equation*} 
    or in other word 
    \begin{equation*}
    	\bx_1 R^{-1} Q^T = \vx^T.
    \end{equation*}
    The QR-algorithm is also Lipschitz-continuous with respect to perturbations in the matrix entries, that with respect to $\vp$, as long as the single vectors do not get linearly dependent. Therefore
    $Q$ and $R$ are Lipschitz-continuous as well, and therefore, since $R$ is a triangular matrix, also $R^{-1}$. 
    This guarantees a local uniformly bound and local Lipschitz-continuity for $\Psi$.
\end{itemize}
	\item The boundedness of $P_\vp$ is obvious because it is an orthogonal projector. 
\end{itemize}
\end{proof}
In the following we study a Gauss-Newton's method for solving \autoref{eq:ip_2}, where $\Psi: \domain{\Psi} \subseteq \R^{n_*} \to \bX$ is a Lipschitz-continuous immersion (see \autoref{de:immersion}), 
$F:\bX \to \bY$ is bounded and $N=F \circ \Psi$.

\begin{lemma} \label{le:dec} Let $F: \bX \to \bY$ be linear, bounded, with trivial nullspace and dense range. Moreover, let $\Psi:\domain{\Psi} \subseteq \R^{n_*} \to \bX$ be a Lipschitz-differentiable immersion; see \autoref{de:immersion}, and let $N = F \circ \Psi$. 
	
	Then $\domain{N}=\domain{\Psi}$ and for every 
	$\vp \in \domain{N}$ the derivative of the operator $N$ at a point $\vp$ has a Moore-Penrose inverse $N'(\vp)^\dagger$, 
	which satisfies:
	\begin{itemize}
		\item Decomposition property of the Moore-Penrose inverse:
		\begin{equation} \label{eq:MPa}
			N'(\vp)^\dagger \bz = \Psi'(\vp)^\dagger F^{-1} \bz \text{ for all } \vp \in \domain{N}, \bz \in \range{F} \subseteq \bY.
		\end{equation}
		In particular this means that
		\begin{equation}\label{eq:MPb}
			N'(\vp)^\dagger N'(\vp) = I \text{ on } \R^{n_*} \text{ and } 
			N'(\vp) N'(\vp)^\dagger = Q|_{\range{FP_\vp}},
		\end{equation}
	    where $I$ denotes the identity operator on $\R^{n_*}$ and $Q: \bY \to \overline{\range{FP_\vp}} \dot{+} \range{FP_\vp}^\bot$, respectively.
		\item Generalized Newton-Mysovskii condition:
		\begin{equation} \label{eq:wrnm} \begin{aligned}
			\norm{N'(\vp)^\dagger(N'(\vq+s(\vp-\vq)-N'(\vq))(\vp-\vq)}_{\bP}  
			&\leq s C_I C_L \norm{\vp-\vq}_\bP^2\\
			& \quad \vp, \vq \in \domain{N}, s \in [0,1]\;.
			\end{aligned}
		\end{equation}
	    We recall that the Lipschitz-constants $C_I$ and $C_L$ are defined in \autoref{eq:cl}.
	\end{itemize}
\end{lemma}
\begin{proof}
	First of all, we note that
	\begin{equation*}
		N'(\vp) = F \Psi'(\vp) \text{ on } \domain{\Psi}=\domain{N}.
	\end{equation*}
	To prove \autoref{eq:MPa} we have to verify \autoref{eq:MP} with 
	$$L := N'(\vp) = F \psi'(\vp): \bP \to \bY \text{ and } B := \Psi'(\vp)^\dagger F^{-1} : \range{F} \subseteq \bY \to \bP.$$
	Note that since we assume that $F$ has dense range we do not need to define and consider $B$ on $\range{F} \dot{+} \underbrace{\range{F}^\bot}_{=\set{0}}$.
	
	Let us first state that with the notation of \autoref{eq:MP} we have for fixed $\vp$:
	\begin{equation*}
		\domain{B}=\domain{\Psi'(\vp)^\dagger F^{-1}}=\range{F} \text{ and } 
		\range{L} = \set{F \Psi'(\vp) \vq : \vq \in \R^{n_*}} = \range{F P_\vp}.
	\end{equation*}
    We use $P \equiv 0$ in \autoref{eq:MP}.
    
    In particular the first item shows that for $\bz = F \bx = F P_\vp \bx + F (I-P_\vp)\bx$ we have 
    \begin{equation} \label{eq:id4}
    	Q \bz = Q(F P_\vp \bx + F (I-P_\vp)(\bx))= F P_\vp \bx.
    \end{equation}
	
	 Applying \autoref{le:MPN} and the invertability of $F$ on the range of $F$ shows that  
	\begin{equation*} 
		\begin{aligned}
			LBL=F \Psi'(\vp)\Psi'(\vp)^\dagger F^{-1} F \Psi'(\vp)&=F \Psi'(\vp)\Psi'(\vp)^\dagger \Psi'(\vp) \underbrace{=}_{\autoref{eq:id1}} F \Psi'(\vp) =L,\\
			BLB = \Psi'(\vp)^\dagger F^{-1} F \Psi'(\vp) \Psi'(\vp)^\dagger F^{-1}&=\Psi'(\vp)^\dagger \Psi'(\vp) \Psi'(\vp)^\dagger F^{-1} \underbrace{=}_{\autoref{eq:id2a}} \Psi'(\vp)^\dagger F^{-1}=B,\\
			BL &=\Psi'(\vp)^\dagger F^{-1} F \Psi'(\vp) \underbrace{=}_{\autoref{eq:id1}} I - P = I \text{ on } \R^{n_*}.
		\end{aligned}
	\end{equation*}
    For the forth identity we take some $\bz \in \range{F}$. Therefore, there exists some $\bx \in \bX$ such that $F\bx=\bz$ and thus
    \begin{equation*}
    	\begin{aligned}
    		LB\bz &= F \Psi'(\vp)\Psi'(\vp)^\dagger F^{-1} \bz = F \Psi'(\vp)\Psi'(\vp)^\dagger \bx
    	 \underbrace{=}_{\autoref{eq:id2}} F P_\vp \bx \underbrace{=}_{\autoref{eq:id4}} Q\bz.
    	\end{aligned}
    \end{equation*}
	Thus \autoref{eq:MP} holds and we know that $N'(\vp)^\dagger$ from \autoref{eq:MPa} is the Moore-Penrose of $N'(\vp)$. The last two identities in fact show \autoref{eq:MPb}.
	
	From \autoref{eq:cl} it follows that
	\begin{equation*} \begin{aligned} \norm{N'(\vp)^\dagger(N'(\vq+s(\vp-\vq)-N'(\vq))(\vp-\vq)}_{\bP} 
			= & \norm{\Psi'(\vp)^\dagger F^{-1}(F\Psi'(\vq+s(\vp-\vq)-F\Psi'(\vq))(\vp-\vq)}_{\bP} \\=&
			\norm{\Psi'(\vp)^\dagger (\Psi'(\vq+s(\vp-\vq)-\Psi'(\vq))(\vp-\vq)}_{\bP} \\
			\leq & C_IC_L s \norm{\vp-\vq}_\bP^2
			\text{ for all } \vp, \vq \in \domain{\Psi} = \domain{N}\;,
		\end{aligned}
	\end{equation*}
    thus \autoref{eq:wrnm}.
\end{proof}
We have now all ingredients to prove a local convergence rates result for a Gauss-Newton's method, where the operator $N$ is the  composition of a linear bounded operator and a Lipschitz-differentiable immersions:
\begin{theorem} \label{th:deupot92dg} Let $F: \bX \to \bY$ be linear, bounded, with trivial nullspace and dense range. Moreover, let 
	$\Psi: \domain{\Psi} \subseteq \bP \to \bX$ be a Lipschitz-differentiable immersion with $\domain{\Psi}$ open, non-empty, and convex. Moreover, $N= F \circ \Psi : \domain{\Psi} \to \bY$. We assume that there exist $\vp^\dagger \in \domain{\Psi}$ that satisfies
	\begin{equation} \label{eq:sol}
		N(\vp^\dagger) = \by.
	\end{equation}
	Moreover, we assume that there exists $\vp^0 \in \domain{\Psi}$, which satisfies \autoref{eq:h}.
	Then, the iterates of the Gauss-Newton's iteration,
	\begin{equation} \label{eq:newton} \begin{aligned}
			\vp^{k+1} = \vp^k - N'(\vp^k)^\dagger(N(\vp^k)-\by) 
			\quad k \in \N_0
		\end{aligned}
	\end{equation}
	are well-defined elements in $\overline{\mathcal{B}(\vp^0,\rho)}$ and converge quadratically to $\vp^\dagger$.
\end{theorem}
\begin{proof}
	First of all note, that $\domain{\Psi} = \domain{N}$ since $F$ is defined all over $\bX$. 
	
	Let $\rho = \norm{\vp^\dagger-\vp^0}_\bP$:
	We prove by induction that $\vp^k \in \overline{\mathcal{B}(\vp^\dagger;\rho)}$ for all $k \in \N_0$.
	\begin{itemize}
		\item For $k=0$ the assertion is satisfied by assumption \autoref{eq:h}.
		\item Let $\vp^k \in \overline{\mathcal{B}(\vp^\dagger;\rho)}$. Using the first condition of \autoref{eq:MP}, which a Moore-Penrose inverse satisfies, we see that
		\begin{equation*}
			N'(\vp^k) N'(\vp^k)^\dagger N'(\vp^k) (\vp^{k+1}-\vp^\dagger) =
			N'(\vp^k)(\vp^{k+1}-\vp^\dagger). 
		\end{equation*}
		The definition of Gauss-Newton's method, \autoref{eq:newton}, and \autoref{eq:sol} then imply that
		\begin{equation*}
			N'(\vp^k) (\vp^{k+1}-\vp^\dagger) =
			N'(\vp^k) N'(\vp^k)^\dagger (N(\vp^\dagger) - N(\vp^k) - N'(\vp^k) (\vp^\dagger-\vp^k)),
		\end{equation*}
		and consequently, using the third identity of \autoref{eq:MP} (note that under the assumptions of this theorem $P=0$, see the proof prior to \autoref{eq:id1}), the second identity of \autoref{eq:MP} and that $F$ is injective, we get 
		\begin{equation*}
			\begin{aligned}
				\vp^{k+1}-\vp^\dagger = N'(\vp^k)^\dagger N'(\vp^k) (\vp^{k+1}-\vp^\dagger) &=
				N'(\vp^k)^\dagger (N(\vp^\dagger) - N(\vp^k) - N'(\vp^k) (\vp^\dagger-\vp^k))\\
				&=
				\Psi'(\vp^k)^\dagger (\Psi(\vp^\dagger) - \Psi(\vp^k) - \Psi'(\vp^k) (\vp^\dagger-\vp^k)).
			\end{aligned}
		\end{equation*}
		From the Newton-Mysovskii condition \autoref{eq:wrnm} and \autoref{eq:h} it then follows that 
		\begin{equation}\label{eq:help}
			\begin{aligned}
				\norm{\vp^{k+1}-\vp^\dagger}_\bP \leq \frac{C_IC_L}{2} \norm{\vp^k - \vp^\dagger}^2_\bP 
				\leq \frac{C_IC_L\rho}{2} \norm{\vp^k - \vp^\dagger}_\bP < \norm{\vp^k - \vp^\dagger}_\bP \\
				\text{ or } 
				\norm{\vp^{k+1}-\vp^\dagger}_\bP = \norm{\vp^k - \vp^\dagger}_\bP=0.
			\end{aligned}
		\end{equation}
		This, in particular shows that $\vp^{k+1} \in \mathcal{B}(\vp^\dagger;\rho)$, thus the well-definedness of the Gauss-Newton's iterations in the 
		closed ball.
	\item 
	Using \autoref{eq:help} we then get, since $h = C_I C_L \rho/2 < 1$, that
	\begin{equation*}
		\norm{\vp^{k+1}-\vp^\dagger}_\bP \leq h^{k+1} \norm{\vp^0 - \vp^\dagger}_\bP \leq h^{k+1} \rho,
	\end{equation*}
	which converges to $0$ for $k \to \infty$. 
	\item Convergence and the first inequality of \autoref{eq:help} imply quadratic convergence.
	\end{itemize}
\end{proof}
\begin{remark} Based on the assumption of an immersion we have shown in \autoref{le:dec} that $\Psi(\vp)^\dagger F^{-1}$ is the Moore-Penrose inverse of $N=F \Psi(\vp)$.
	In order to prove (quadratic) convergence of Gauss-Newton's methods one only requires an \emph{outer inverse} (see \autoref{not:inverse}). Following \cite{NasChe93} (see also \cite{Hae86}) the analysis of Gauss-Newton's method could be based on \emph{outer inverses}, which is more general than for the Moore-Penrose inverse (compare \autoref{eq:outer} and \autoref{eq:MP}). However, it is nice to actually see that $N'(\vp)^\dagger$ is a Moore-Penrose inverse, which is the novelty compared to the analysis of \cite{NasChe93}. 
	For excellent expositions on Kantorovich and Mysovskii theory see \cite{KanAki64,OrtRhe70,Schw79} - here we replace the Newton-Mysovskii conditions by properties of an immersion. For aspects related to Newton's methods for singular points see \cite{DecKelKel83,Gri85}. For applications of generalized inverses in nonlinear analysis see 
	\cite{Nas87,Nas79}.
\end{remark}

\subsection{Neural networks}
We want to apply the decomposition theory to Gauss-Newton's methods for solving \autoref{eq:ip_2}, where 
$\Psi$ is a \emph{shallow neural network operator}. 
\begin{definition}[Shallow neural network operator] \label{de:snn}
	Let $N \in \N$ be fixed. We consider the operator 
	\begin{equation}\label{eq:classical_approximation}
		\begin{aligned}
			\Psi : \bP:=\R^N \times \R^{n \times N} \times \R^N &\to C^1([0,1]^n) \subseteq \bX := L^2([0,1]^n) ,\\ 
			(\val ,\bw,\vth ) &\mapsto \left(\vx \to \sum_{j=1}^{N} \alpha_j\sigma \left( \bw_j^T \vx +\theta_j\right) \right) \\
			\text{ where } \alpha_j, \theta_j \in \R \text{ and } \vx, \bw_j \in \R^{n}.
		\end{aligned}
	\end{equation}
    Note, that with our previous notation, for instance in \autoref{de:immersion}, we have $n_*= (n+2)*N$.
\end{definition}
We summarize the notation, because it is quite heavy:
\begin{enumerate}
	\item $\vec{\cdot}$ denotes a vector in $\R^n$ or $\R^N$,
	\item ${\bf{w}}$ denotes a matrix: The only exception is \autoref{de:deepnl}, where it is a tensor. ${\bf{w}}_j$ denotes a vector, aside from \autoref{de:deepnl}, where it is again a tensor. 
\end{enumerate}

\begin{example}[Examples of activation functions]
    $\sigma$ is called the \emph{activation function}, such as
    \begin{itemize}
        \item the \emph{sigmoid function}, defined by 
              \begin{equation} \label{eq:sigmoid} 
                 \sigma (t) = \frac{1}{1+\e^{-\frac{1}{\ve} t}} \text{ for all } 
                 t \in \R.
              \end{equation} 
          Note, we omit the $\ve$ dependence for notational convenience. 
        \item The \emph{hyperbolic tangent}
        \begin{equation} \label{eq:tanh} 
        	t \to \tanh(t) = \frac{\e^{2t}-1}{\e^{2t}+1}.
        \end{equation}
        \item The \emph{ReLU} activation function,
        \begin{equation} \label{eq:relu} 
                 \sigma(t) = \max\set{0,t} \text{ for all } 
                 t \in \R.
              \end{equation}      
      \item The \emph{step function}, which is the pointwise limit of the sigmoid function, with respect to $\ve \to 0$,
          \begin{equation} \label{eq:step} 
          \sigma(t) = \left\{\begin{array}{rl} 
          	0 & \text{ for } t < 0,\\
          	\frac{1}{2} & \text{ for } t=0,\\
          	1 & \text{ for } t > 0.
          \end{array} \right. 
          t \in \R.
      \end{equation}
    \begin{figure}[h]
      \begin{center}
      	\includegraphics[width=0.3\textwidth]{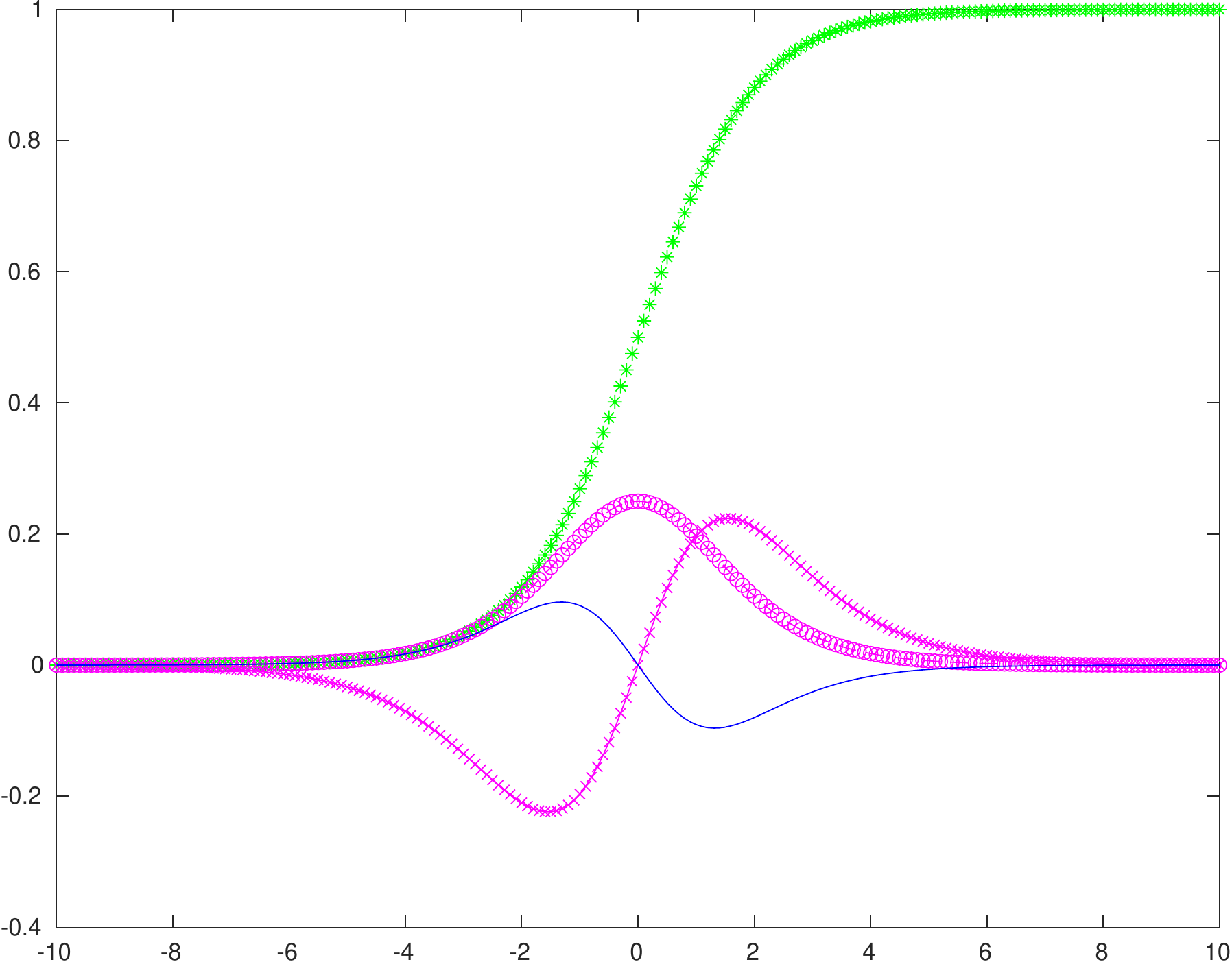} \includegraphics[width=0.3\textwidth]{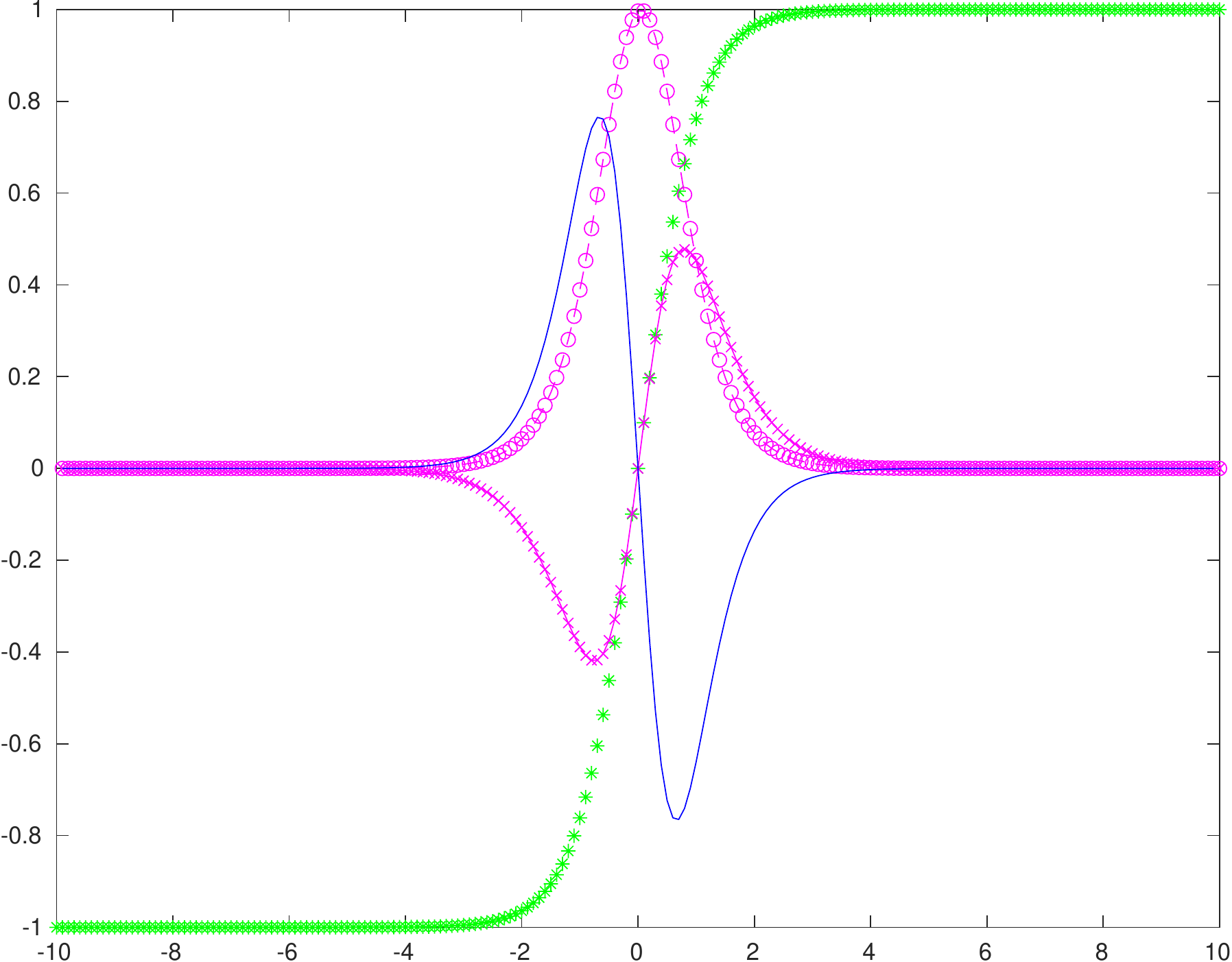}\includegraphics[width=0.3\textwidth]{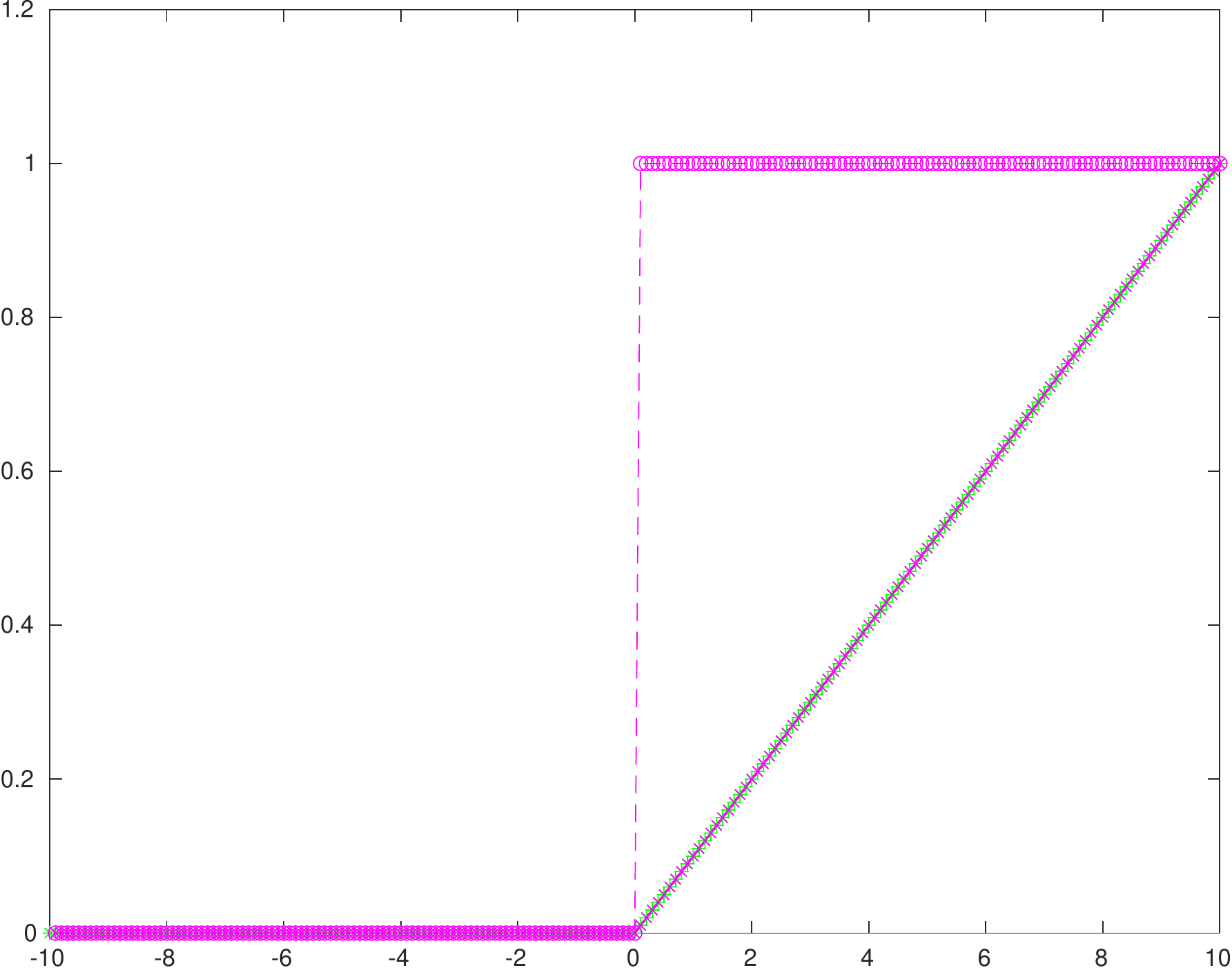}
      	\caption{\label{fig:scheme} Three different activation functions: Sigmoid, tanh and ReLU. 
      		O-th derivative green, first derivative pink with circles and with pink $\times$ is the derivative multiplied by $x$, 2nd derivative blue.
      		The ReLU function is scaled by a factor $1/10$ and the derivative is plotted in original form, derivative times $x$ is again scaled by $1/10$ for visualization purposes.}
      \end{center}
  \end{figure}
    \end{itemize}
	\end{example}
	We only consider shallow neural networks in contrast to \emph{deep neural networks}, which consist of several layers of shallow neural networks (see for instance \cite{Hor91}):
	\begin{definition}[Deep neural networks] \label{de:deepnl}
	Let 
	\begin{equation*}
		\bP_l:= \R^{N_l} \times \R^{n \times N_l} \times \R^{N_l} \text{ for } l =1,\ldots,L \text{ and } 
		\bP:= \prod_{l=1}^L \bP_l.
	\end{equation*}
	Then a deep neural network consisting of $L$ layers is written as 
	\begin{equation}\label{eq:deep_approximation}
		\begin{aligned}
			\Psi : \bP &\to  L^2([0,1]^n) ,\\ 
			(\val _l,\bw_l,\vth _l)_{l=1}^L &\mapsto 
			\left(\vx \to \sum_{j_L=1}^{N_L} \alpha_{j_L}^L \sigma_{\ve_L}^L\left( p_{j_L,L} \left( \sum_{j_{L-1}=1}^{N_{L-1}} \cdots \left( \sum_{j_1=1}^{N_1} \alpha_{j_1,1} \sigma_{\ve_1}^1 \left(p_{j_1}^1(\vx) \right) \right)\right) \right)\right), 
		\end{aligned}
	\end{equation}	
	where 
	$$p_j^i(\vx) = \bw_{j,i}^T \vx +\theta_j^i 
	\text{ with } \alpha_j^i, \theta_j^i \in \R \text{ and } \vx, \bw_{j,i} \in \R^{n} \text{ for all } i=1,\ldots,L.$$
	
	Note that the values $\ve_k$, $k=1,\ldots,L$ can be chosen differently for activation functions at different levels (cf. \autoref{eq:sigmoid}).
\end{definition}
	The success of neural network is due to the universal approximation properties, proven for the first time in \cite{Cyb89,HorStiWhi89}. The universal approximation result states that shallow neural networks are universal, that is, that each continuous function can be approximated arbitrarily well by a neural network function. We review this result now.

\begin{theorem}[\cite{Hor91}]\label{le:LinearApproximation} In dependence of the smoothness of the activation function $\sigma$ there exist two classes of results.
	\begin{itemize}
		\item Theorem 2 from \cite{Hor91}: Let $\sigma :\R \to \R^+$ be a {\bf continuous, bounded and nonconstant} function. Then, for
		every function $g \in C(\R^n)$ and every $\nu>0$, there exists a function
		\begin{equation}\label{eq:Linear}
			\vx \to G(\vx) = \sum_{j=1}^{N}\alpha_j \sigma(\bw_j^T\vx + \theta_j) \qquad \text{ with } N \in \N, \alpha_j, \theta_j \in \R, \bw_j \in \R^n,
		\end{equation}
		satisfying
		\begin{equation*}
			|G(\vx)-g(\vx)|<\nu \text{ uniformly for all compact subsets } K \subseteq \R^n.
		\end{equation*}
	   \item Theorem 1 from \cite{Hor91}: Let $\sigma :\R \to \R^+$ be {\bf unbounded and nonconstant}. Then for every measure $\mu$ on $\R^n$ and every constant $\nu > 0$ and $p \geq 1$ there exists  a function $G$ of the form \autoref{eq:Linear} that satisfies 
	   \begin{equation*}
	    	\int_{\R^n} |G(\vx)-g(\vx)|^p d \vx < \nu.
	    \end{equation*}
	\end{itemize}
\end{theorem}
The first result applies for instance to the sigmoid and hyperbolic tangent function (see \autoref{eq:sigmoid} and \autoref{eq:tanh}).
The second result applies to the ReLU function (see \autoref{eq:relu}). In particular all approximation properties also hold on the compact set $[0,1]^n$, which we are considering. 

\subsection{Newton-Mysovskii condition with neural networks} \label{sec:cond}
In the following we verify Newton-Mysovskii conditions for $\Psi$ being the encoder of \autoref{eq:classical_approximation}.
First we calculate the first and second derivatives of $\Psi$ with respect to $\val ,\bw$ and $\vth $. The computations can be performed for deep neural network encoders as defined in \autoref{eq:deep_approximation} in  principle analogously, but are technically and notationally more complicated.
To make the notation consistent we define 
\begin{equation*}
	\vp := (\val ,\bw,\vth ) \in \R^N \times \R^{n\times N} \times \R^N = \R^{n_*}.
\end{equation*}

\begin{lemma} \label{le:sigma}
	Let $\sigma : \R \to \R^+$ be a two times differentiable function with uniformly bounded 
	function values and first, second order derivatives, such as the sigmoid, hyperbolic tangent functions (see \autoref{fig:scheme})\footnote{This assumption is actually too restrictive, and only used to see that $\Psi \in L^2([0,1]^n)$.}. Then, the derivatives of $\Psi$ 
	with respect to the coefficients $\vp$ 
	are given by the following formulas:
    \begin{itemize}
    	\item Derivative with respect to $\alpha_s$, $s=1,\ldots,N$:
	\begin{equation}\label{eq:ca_d1}
		\begin{aligned}
			\frac{\partial \Psi}{\partial \alpha_s}[\vp](\vx) &= \sigma \left(\sum_{i=1}^n w_s^i x_i +\theta_s \right) \text{ for } s=1,\ldots,N.
		\end{aligned}
	\end{equation}	
\item Derivative with respect to $w_s^t$ where $s=1,\ldots,N$, $t=1,\ldots,n$:
\begin{equation}\label{eq:ca_d2}
\begin{aligned}
\frac{\partial \Psi}{\partial w_s^t} [\vp](\vx) &= \sum_{j=1}^{N} \alpha_j\sigma' \left(\sum_{i=1}^n w_j^i x_i +\theta_j \right) \delta_{s=j}x_t = \alpha_s \sigma' \left(\sum_{i=1}^n w_s^i x_i +\theta_s \right) x_t \\
&
\end{aligned}
\end{equation}
\item Derivative with respect to $\theta_s$ where $s=1,\ldots,N$: 
\begin{equation}\label{eq:ca_d3}
	\begin{aligned}\frac{\partial \Psi}{\partial \theta_s} [\vp](\vx)&= \sum_{j=1}^{N} \alpha_j\sigma' \left(\sum_{i=1}^n w_j^i x_i +\theta_j \right) \delta_{s=j} = \alpha_s \sigma' \left(\sum_{i=1}^n w_s^i x_i +\theta_s \right).
\end{aligned}
\end{equation}
\end{itemize}
Note, that all the derivatives above are functions in $\bX = L^2([0,1]^n)$.  
In particular, maybe in a more intuitive way, we have
	\begin{equation} \label{eq:DG}
		D \Psi[\vp](\vx)  \vh
		= \begin{pmatrix}
			\frac{\partial \Psi}{\partial \val }[\vp](\vx) &  
			\frac{\partial \Psi}{\partial \bw}[\vp](\vx) &\frac{\partial \Psi}{\partial \vth }[\vp](\vx)
		\end{pmatrix}^T\vh
		 \text{ for all } \vh = \begin{pmatrix} \vh_{\val } \\ \bh_\bw \\ \vh_{\vth }
		 \end{pmatrix}\in \R^{n_*} \text{ and } \vx \in \R^n.
	\end{equation}
	Moreover, let $s_1,s_2=1,\ldots,N$, $t_1,t_2 = 1,\ldots,n$, then we have in a formal way:
	\begin{equation}\label{eq:ca_dd}
		\begin{aligned}
			\frac{\partial^2 \Psi}{\partial \alpha_{s_1} \partial \alpha_{s_2}}(\vx) &= 0,\\
			\frac{\partial^2 \Psi}{\partial \alpha_{s_1} \partial w_{s_2}^{t_1}}(\vx) &= \sigma' \left(\sum_{i=1}^n w_{s_1}^i x_i +\theta_{s_1} \right) x_{t_1} \delta_{s_1=s_2}, \\
			\frac{\partial^2 \Psi}{\partial \alpha_{s_1} \partial \theta_{s_2}} (\vx)&= \sigma' \left(\sum_{i=1}^n w_{s_1}^i x_i +\theta_{s_1} \right) \delta_{s_1=s_2}, \\
			\frac{\partial^2 \Psi}{\partial w_{s_1}^{t_1} \partial w_{s_2}^{t_2}} (\vx)&= 
			\alpha_{s_1} \sigma'' \left(\sum_{i=1}^n w_{s_1}^i x_i +\theta_{s_1} \right) x_{t_1}x_{t_2} \delta_{s_1=s_2}, \\
			\frac{\partial^2 \Psi}{\partial w_{s_1}^{t_1} \partial \theta_{s_2}}(\vx) &= 
			\alpha_{s_1} \sigma'' \left(\sum_{i=1}^n w_{s_1}^i x_i +\theta_{s_1} \right) x_{t_1} \delta_{s_1=s_2}, \\
			\frac{\partial^2 \Psi}{\partial \theta_{s_1} \partial \theta_{s_2}}(\vx) &= \alpha_{s_1} \sigma'' \left(\sum_{i=1}^n w_{s_1}^i x_i +\theta_{s_1} \right)\delta_{s_1=s_2},
		\end{aligned}
	\end{equation}
	where $\delta_{a=b} = 1$ if $a=b$ and $0$ else, that is the Kronecker-delta. 
\end{lemma}
The notation of directional derivatives with respect to parameters might be confusing. Note, that for instance 
$\frac{\partial \Psi}{\partial \theta_s} [\vp](\vx)$ denotes a direction derivative of the functional $\Psi$ with respect to the variable $\theta_s$ and this derivative is a function, which depends on $\vx$. The argument, where the derivative is evaluated is a vector. So in such a formula $\theta_s$ has two different meanings. Notationaly differentiating between them would be exact but becomes quite unreadable.

\begin{remark}
\begin{itemize}
	\item 	In particular \autoref{eq:ca_dd} shows that 
		\begin{equation} \label{eq:d2}
			\begin{pmatrix} 
				\vh_{\val } & \bh_\bw &\vh_{\vth }
			\end{pmatrix}
				D^2 \Psi [\vp](\vx)
			\begin{pmatrix} 
				\vh_{\val } \\ \bh_\bw \\ \vh_{\vth }
			\end{pmatrix} 
			\text{ is continuous (for fixed $\vx$) with respect to }\vp.
		\end{equation}
	\item 
We emphasize that under the assumptions of \autoref{le:sigma} the linear space (for fixed $\vp$) 
\begin{equation*}
	\range{D\Psi[\vp]} = \set{D\Psi [\vp] \vh : \vh = (\vh_{\val },\bh_\bw,\vh_{\vth }) \in \R^{N \times (n+2)}} \subseteq L^2([0,1]^n). 
\end{equation*}
\item In order to prove convergence of the Gauss-Newton's method, \autoref{eq:newton}, by applying \autoref{th:deupot92}, we have to prove that $\Psi$ is a Lipschitz-continuous immersion. Below 
we lack proving one important property so far, namely, that
	\begin{equation} \label{eq:co1}
		\partial_k \Psi[\vp],
	\quad k=1,\ldots,n_*=N(n+2)
\end{equation}
are linearly independent functions. In this paper, this will remain open as a conjecture, and the following statements are 
valid modulo this conjecture.
\end{itemize}
\end{remark}
In the following we survey some results on linear independence with respect to the coefficients $\val ,\bw,\vth $ of the functions 
$\vx \to \sigma \left(\sum_{i=1}^n w_s^i x_i +\theta_s \right)$, which match the functions  
$\vx \to \frac{\partial \Psi}{\partial \alpha_s}[\vp](\vx)$, that is with respect to the first $N$ variables.

\subsection{Linear independence of activation functions and its derivatives}

The universal approximation results from for instance \cite{Cyb89,HorStiWhi89,Hor91} do not allow to conclude 
that neural networks function as in \autoref{eq:classical_approximation} are linearly independent. 
Linear independence is a non-trivial research question: 
We recall a result from \cite{Lam22} from which linear independence of a shallow neural network operator, 
as defined in \autoref{eq:classical_approximation}, can be deduced for a variety of activator functions. 
Similar results on linear independence of shallow network functions based on sigmoid activation functions have been stated in \cite{TamTat97,Gua03}, but the discussion in \cite{Lam22} raises questions on the completeness of the proofs. 
In \cite{Lam22} it is stated that all activation functions from the \emph{Pytorch library} \cite{PasGroMasLerBra19} are linearly independent with respect to almost all parameters $\bw$ and $\theta$.
\begin{theorem}[\cite{Lam22}] \label{th:lam_main} For all activation functions \emph{HardShrink, HardSigmoid, HardTanh, HardSwish, LeakyReLU,
	PReLU, ReLU, ReLU6, RReLU, SoftShrink, Threshold}, \emph{LogSigmoid, Sigmoid, SoftPlus, Tanh, and TanShrink} and 
    the \emph{PyTorch} functions \emph{CELU, ELU, SELU} the shallow neural network functions \autoref{eq:classical_approximation} formed by randomly generated vectors $(\bw,\vth )$ are linearly independent.
\end{theorem}
\begin{remark}
	\begin{enumerate} 
		\item \autoref{th:lam_main} states that the functions $\frac{\partial \Psi}{\partial \alpha_s}$ (taking into account \autoref{eq:ca_d1}) are linearly independent for \emph{almost all} parameters $(\bw,\vth ) \in \R^{n \times N} \times \R^N$. In other words, the first block of the matrix is $D \Psi$ in \autoref{eq:DG} consists of functions, which are linearly independent for almost all parameters $(\bw,\vth )$. 
		For our results to hold we need on top that the functions $\frac{\partial \Psi}{\partial w_s^t}$ and $\frac{\partial \Psi}{\partial \theta_s}$ 
		from the second and third block (see \autoref{eq:ca_dd}) are linearly independent within the blocks, respectively, and also across the blocks. So far this has not been proven but can be conjectured already from \autoref{fig:scheme}. 
		\item For the sigmoid function we have \emph{obvious symmetries} because 
		\begin{equation} \label{eq:antisymmetric}
			\sigma' \left(\bw_j^T \vx +\theta_j \right) = \sigma' \left(-\bw_j^T \vx -\theta_j \right) \text{ for every } \bw_j \in \R^n, \vth  \in \R^N,
		\end{equation}
		or in other words for the function $\Psi$ from \autoref{eq:classical_approximation} we have according to \autoref{eq:ca_d3} that 
		\begin{equation}\label{eq:antisym}
			\frac{\partial \Psi}{\partial \theta_s} [\val ,\bw,\vth ](\vx)
			=  \alpha_s \sigma'(\bw_j^T \vx + \theta_j) 
			=  \alpha_s \sigma'(-\bw_j^T \vx - \theta_j) = 
			\frac{\partial \Psi}{\partial \theta_s}[\val ,-\bw,-\vth ](\vx)
		\end{equation}
	    or in other words $\frac{\partial \Psi}{\partial \theta_s}[\val ,\bw,-\vth ]$ and 
	    $\frac{\partial \Psi}{\partial \theta_s}[\val ,-\bw,-\vth ]$ are linearly dependent.
	\end{enumerate}
\end{remark}
\begin{conjecture}\label{co:independence} We define by $\domain{\Psi}$ a \emph{maximal set of vectors} $(\val ,\bw,\vth )$ such that 
	the $n_*=N\times(n+2)$ functions in $\vx$ 
	\begin{equation*}
	\vx \to \frac{\partial \Psi}{\partial \alpha_s}[\val , \bw,\vth ](\vx), \quad
	\vx \to \frac{\partial \Psi}{\partial w_s^t}[\val ,\bw,\vth ](\vx), \quad 
	\vx \to \frac{\partial \Psi}{\partial \theta_s}[\val ,\bw,\vth ](\vx), \quad s=1,\ldots,N, t = 1,\ldots,n,
	\end{equation*}
    are linearly independent. We assume that $\domain{\Psi}$ is open and dense in $\bX \in L^2([0,1])^2$ The later is guaranteed by \autoref{le:LinearApproximation}. Recall the discussion above: The differentiation variables and the arguments coincide notationally, but are different objects.
\end{conjecture}
\begin{remark}
\begin{itemize}
\item It can be conjectured that for every element from $\domain{\Psi}$ only one element in $\R^{n_*}$ exists, 
which satisfies \emph{obvious symmetries} such as formulated in \autoref{eq:antisym}. These ``mirrored'' elements 
are a set of measure zero in $\bP$. We conjecture that this corresponds to the set of measure zero as stated in \cite{Lam22}, which is derived with Fourier methods.
\item \autoref{eq:ca_dd} requires that all components of the vector $\val $ are non-zero. This means in particular that for ``sparse solutions'', with less that $n_*=N*(n+2)$ coefficients,
convergence is not guaranteed, because of a locally degenerating submanifold. 
We consider the manifold given by the function 
\begin{equation} \label{eq:vis}
	\begin{aligned}
		F:\R^2 &\to \R^2.\\
		\begin{pmatrix}x\\y\end{pmatrix} &\mapsto \begin{pmatrix}xy\\x^2+y^2\end{pmatrix}
	\end{aligned}
\end{equation}
Then 
\begin{equation*}
		\nabla F (x,y) = 
		\begin{pmatrix} y & x\\ 2x & 2y\end{pmatrix}.
\end{equation*}
We have $\det \nabla F (x,y) = 2(y^2 - x^2)$, which vanishes along the diagonals in $(x,y)$-space. That is of the diagonals the function is locally a submanifold (see \autoref{fig:vis}): 
\begin{figure}[h]
\begin{center}
\includegraphics[width=0.5\textwidth]{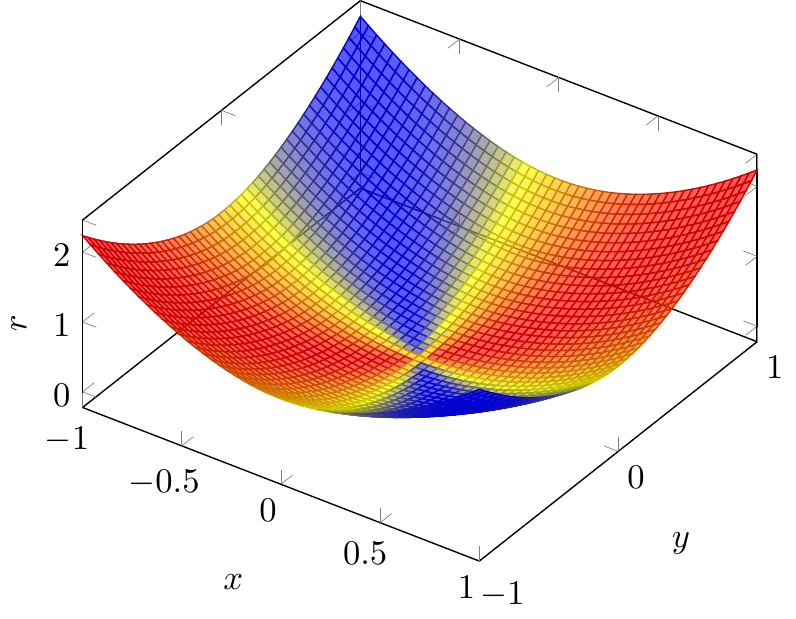}
\caption{\label{fig:vis} 
The function $F$ from \autoref{eq:vis}. We have plotted $F(x,y)$ via its polar coordinates, I.e. $r=\abs{F(x,y)}$ and $\theta=\tan^{-1}\left( \frac{xy}{x^2+y^2} \right)$. The colors correspond to identical angles.  
}
\end{center}
\end{figure}
\end{itemize}
\end{remark}

\subsection{Local convergence of Gauss-Newton's method with coding networks}
In the following we prove a local convergence result for a Gauss-Newton's method, for solving operator equations 
\autoref{eq:ip_2} where $F$ is complemented by a shallow neural network coder $\Psi$. In order to apply \autoref{th:deupot92} we have to verify that the shallow neural network operator (see \autoref{eq:classical_approximation}) is a Lipschitz-differentiable immersion. 

\begin{lemma} \label{le:shallow}Let $F:\bX = L^2([0,1]^n) \to \bY$ be linear, bounded, with trivial nullspace and closed range, and let $\sigma$ be strictly monotonic (like sigmoid or hyperbolic tangent) and satisfy the assumptions of \autoref{le:sigma}. Moreover, assume that \autoref{co:independence} holds.
	Then 
	\begin{enumerate}
		\item For every element $\vp = (\val ,\bw,\vth ) \in \R^{n_*}$ in the maximal set 
		$\domain{\Psi}$ (see \autoref{co:independence}), 
		$\mathcal{R}(D\Psi[\vp])$ is a linear subspace of the space $\bX $ of dimension $n_*=N\times(n+2)$.
		\item There exists an open neighborhood $\mathcal{U} \subseteq \R^{N \times (n+2)}$ of vectors $(\val ,\bw,\vth )$ 
		such that $\Psi$ is a Lipschitz-differentiable immersion in $\mathcal{U}$.
	\end{enumerate}
\end{lemma}
\begin{proof}
	\begin{itemize}
		\item It is clear that for each fixed $\vp$, $D\Psi[\vp] \in L^2([0,1]^n)$ because of the differentiability assumptions of $\sigma$, see \autoref{eq:d2}.
		\autoref{co:independence} implies that $\mathcal{R}(D\Psi[\vp])$ is a linear subspaces of $\bX$ of dimension $N\times(n+2)$ (note the elements are functions).
		\item $D^2\Psi[\vp]: \R^{N \times (n+2)} \to L^2([0,1]^n)$ is continuous (see \autoref{eq:d2}) since we assume that the activation function $\sigma$ is twice differentiable. Now we consider a non-empty open neighborhood $\mathcal{U}$ of a vector $\vp$, with a compact closure. 
		\begin{itemize} 
			\item Then, from the continuity of $D^2\Psi$ with respect to $\vp$, it follows that $D\Psi$ is a Fr\`echet-differentiable with Lipschitz-continuous derivative on $\mathcal{U}$. In particular this means that \autoref{it1:immersion} in \autoref{de:immersion} holds. Moreover, \autoref{eq:cl} 
			holds for $\Psi'$. That is, there exists constants $C_L$ and $C_I$ such that		
			\begin{equation} \label{eq:cla}
				\norm{\Psi'(\vp)-\Psi'(\vq)}_{\bP \to \bY} \leq C_L \norm{\vp-\vq}_\bP \text{ and } 			
					\norm{\Psi'(\vp)}_{\bP \to \bY} \leq C_I \text{ for } \vp, \vq \in \domain{\Psi}.
			\end{equation}
		\end{itemize}
	\end{itemize} 
\end{proof}
Note that $\Psi'(p)^\dagger$ as defined in \autoref{eq:MP_Penrose} is also uniformly bounded and Lipschitz-continuous as a consequence of \autoref{le:MPN}.

\begin{theorem}[Local convergence of Gauss-Newton's method] \label{th:newtonNN} Let $F:\bX = L^2([0,1]^n) \to \bY$ be a linear, bounded operator with trivial nullspace and dense range and let $N = F \circ \Psi$, where $\Psi: \domain{\Psi} \subseteq \R^{N \times (n+2)} \to \bX$ is a shallow neural network operator generated by an activation function $\sigma$ which satisfies the assumptions of \autoref{le:shallow} and \autoref{co:independence}.
	Let $\vp^0 \in \domain{\Psi}$ be the starting point of the Gauss-Newton's iteration \autoref{eq:newton} and 
	let $\vp^\dagger \in \domain{\Psi}$ be a solution of \autoref{eq:sol}, which satisfy \autoref{eq:h}.
	Then the Gauss-Newton's iterations are locally, that is if $\vp^0$ is sufficiently close to $\vp^\dagger$, and quadratically converging.
\end{theorem}
\begin{proof}
	The proof is an immediate application of \autoref{le:shallow} to \autoref{th:deupot92dg}.
\end{proof}

\begin{remark}
	We have shown that a nonlinear operator equation, where the operator is a composition of a linear compact operator and a shallow neural network operator, can be solved with a Gauss-Newton's method with guaranteed local convergence in the parameter space. 
\end{remark}

\subsection*{Conclusion} We have shown that Gauss-Newton's methods are efficient algorithms for solving linear 
inverse problems, where the solution can be encoded with a neural network. 
The convergence studies, however, are not complete, and are based on a conjecture on linear independence of 
activation functions and its derivatives.

\subsection*{Acknowledgements}
%
%
This research was funded in whole, or in part, by the Austrian Science
Fund (FWF) P 34981 -- New Inverse Problems of Super-Resolved Microscopy
 (NIPSUM). For the purpose of open access, the author has applied
a CC BY public copyright licence to any Author Accepted Manuscript version arising
from this submission.
Moreover, OS is supported by the Austrian Science Fund (FWF),
with SFB F68 ``Tomography Across the Scales'', project F6807-N36 (Tomography with Uncertainties).
The financial support by the Austrian Federal Ministry for Digital and Economic
Affairs, the National Foundation for Research, Technology and Development and the Christian Doppler
Research Association is gratefully acknowledged.
BH is supported by the German Science Foundation (DFG) under the grant~HO~1454/13-1 (Project No.~453804957).
\section*{References}
\renewcommand{\i}{\ii}
\printbibliography[heading=none]

\end{document}